\documentclass[12pt,oneside,english,a4paper,reqno]{amsart}
\usepackage[T1]{fontenc}
\usepackage[latin9]{inputenc}
\usepackage[a4paper]{geometry}
\usepackage{amstext}
\usepackage{amsthm}
\usepackage{amssymb}
\usepackage{graphicx}
\usepackage{enumitem}
\usepackage{hyperref}

\usepackage{lineno}
    \newcommand*\patchAmsMathEnvironmentForLineno[1]{%
       \expandafter\let\csname old#1\expandafter\endcsname\csname #1\endcsname
       \expandafter\let\csname oldend#1\expandafter\endcsname\csname end#1\endcsname
       \renewenvironment{#1}%
          {\linenomath\csname old#1\endcsname}%
          {\csname oldend#1\endcsname\endlinenomath}}% 
    \newcommand*\patchBothAmsMathEnvironmentsForLineno[1]{%
       \patchAmsMathEnvironmentForLineno{#1}%
       \patchAmsMathEnvironmentForLineno{#1*}}%
    \AtBeginDocument{%
    \patchBothAmsMathEnvironmentsForLineno{equation}%
    \patchBothAmsMathEnvironmentsForLineno{align}%
    \patchBothAmsMathEnvironmentsForLineno{flalign}%
    \patchBothAmsMathEnvironmentsForLineno{alignat}%
    \patchBothAmsMathEnvironmentsForLineno{gather}%
    \patchBothAmsMathEnvironmentsForLineno{multline}%
    }

\makeatletter

\providecommand{\tabularnewline}{\\}

\newtheorem{thm}{Theorem}[section]
\newtheorem{lem}[thm]{Lemma}
\newtheorem{proposition}[thm]{Proposition}
\newtheorem{corollary}[thm]{Corollary}

\newtheorem{conjecture}[thm]{Conjecture}
\newtheorem{definition}[thm]{Definition}
\newtheorem{remark}[thm]{Remark}

\date{}
\usepackage{amsmath}
\usepackage{bbm}
\usepackage{dsfont}

\makeatother

\usepackage{babel}

\newcommand{\emaill}[1]{\email{\href{mailto:#1}{#1}}}
\newcommand{\website}[1]{\urladdr{\href{#1}{#1}}}

\global\long\def\One{\mathds{1}}%

\global\long\def\Laplacian{\Delta}%

\global\long\def\grad{\nabla}%

\global\long\def\norm#1{\left\Vert #1\right\Vert }%

\global\long\def\zz{\mathbb{Z}}%

\global\long\def\rr{\mathbb{R}}%

\global\long\def\nn{\mathbb{N}}%

\global\long\def\pp{\mathbb{P}}%

\global\long\def\ee{\mathbb{E}}%

\global\long\def\floor#1{\left\lfloor #1\right\rfloor }%

\global\long\def\windnum{\mathcal{W}}%

\newcommand{\dd}{\mathrm{d}}
\newcommand{\mZ}{\mathcal{Z}}
\newcommand{\spin}{\mathcal S}
\newcommand{\eqlaw}{\stackrel{(\rm{d})}{=}}

\title{Topologically Induced Metastability in Periodic XY Chain}
\author{Cl\'ement Cosco}
\emaill{clement.cosco@weizmann.ac.il}
\author{Assaf Shapira}
\emaill{assafshap@gmail.com}
\website{https://assafshap.github.io/}

\begin{document}
\begin{abstract}
Non-trivial topological behavior appears in many different contexts in statistical physics, perhaps the most known one being the Kosterlitz-Thouless phase transition in the two dimensional XY model. We study the behavior of a simpler, one dimensional, XY chain with periodic boundary and strong interactions; but rather than concentrating on the equilibrium measure we try to understand its dynamics. The equivalent of the Kosterlitz-Thouless transition in this one dimensional case happens when the interaction strength scales like the size of the system $N$, yet we show that a sharp transition for the dynamics occurs at the scale of $\log N$ -- when the interactions are weaker than a certain threshold topological phases could not be observed over long times, while for interactions that are stronger than that threshold topological phases become metastable, surviving for diverging time scales. 
\end{abstract}
\maketitle


\section{Introduction}

Since the seminal work of Kosterlitz and Thouless \cite{KT73}, physical
systems with non-trivial topological phases have been studied in many
different contexts. One particularly interesting phenomenon studied
in \cite{KT73} is the formation of vortices in the two dimensional
XY model, and the phase transition related to it. The existence of
this phase transition has been proven rigorously in the work of Fr\"ohlich and
Spencer \cite{FroehlichSpencer}; and understanding in depth the topological
behavior of these systems is a very challenging problem.

The purpose of this paper is to study systems with non-trivial topological
phases from a different point of view -- rather than analyzing their
equilibrium measure, we will try to understand their dynamics. For
the sake of this study, we consider a simple, one dimensional, chain
of rotors $X=\left(X_{1},\dots,X_{N}\right)\in\spin^{N}$, where $\spin=\rr/2\pi\zz$.
In order to see a topological effect, we take periodic boundary conditions,
setting $X_{0}=X_{N}$. The interaction between the rotors will be
given by the XY Hamiltonian
\[
H(X_{1},\dots,X_{N})=-J\sum_{i=1}^{N}\cos(X_{i}-X_{i-1}),
\]
which defines the equilibrium measure $\mu$ on $\spin^{N}$:
\begin{equation*}\label{eq:mu_intro}
\text{d}\mu=e^{-\beta H}/\mathcal Z_N\text{d}\lambda,
\end{equation*}
with $\lambda$ the Lebesgue measure.

For fixed inverse temperature $\beta$ and coupling constant $J$
its behavior is trivial, but if $\beta J\rightarrow\infty$ with $N$
we can see different topological phases -- imagine that the interactions
are strong enough, so that the angles between neighboring rotors are all very small. In this case,
the discrete chain looks like a continuous circle -- for each point $t \in \rr/\zz$ we can define $\phi(t) = X_{\floor{N t}}$; and when the interactions are very strong we expect the field $\phi$ to behave like a continuous function.
If this is indeed the case, we can consider topological properties of this field, and in particular, as a continuous function from $\rr/\zz$ to $\spin$ (both homeomorphic to $S^1$), it has a winding number counting how many times the path described by $\phi$ turns around $\spin$. The winding number is a \emph{topological invariant} -- it does not change under continuous deformation of $\phi$ (namely homotopy).
%We can then
%consider the winding number \textcolor{red}{(what is its definition ?)} of the field (as a continuous function
%from $S^{1}$ to $S^{1}$ \textcolor{red}{(to $\mathbb Z$ ?)}). Since it is both integer and continuous
%it must be constant under continuous changes in the configuration,
%thus the system remains in the topological phase defined by its winding
%number \textcolor{red}{(all of that you mean when angles are very small right ?)}.

In fact, the winding number could be defined on a discrete chain $\left(x_{1},\dots,x_{N}\right)\in\spin^N$
whenever $x_{i}-x_{i+1}\neq\pi$ for all $i$ -- for an angle $\theta\in \spin \setminus\left\{ \pi\right\} $,
denote by $\tilde{\theta}$ its representative in $(-\pi,\pi]$ -- 
then the winding number is defined as\footnote{On points where $x_{i}-x_{i+1} = \pi$ for some $i$ the winding number is ill defined; it will sometimes be convenient for us to set $\windnum = \infty$.}
\begin{equation}\label{eq:windnum}
\windnum(x_{1},\dots,x_{N})=\frac{1}{2\pi}\sum_{i=1}^{N}\widetilde{x_{i}-x_{i-1}}.
\end{equation}
This number must be an integer due to the periodic boundary condition,
and it is continuous in its domain of definition. Therefore, it could
only change when at some position $i$ the two rotors $X_{i}$ and
$X_{i+1}$ point in opposite directions. It should be noted that each
phase (defined by a winding number) contains a unique minimum of $H$
at $X=\left(\frac{2\pi\windnum}{N}i\right)_{i=1}^{N}$, and that the
energy difference between these minima is of order $\frac{1}{N}$
(for fixed $\windnum$).

In this paper we will focus on the reversible process with respect
to the measure $\mu$ given by the It\^o equation 
\begin{equation} \label{eq:diffusion}
\text{d}X (t)=-\grad H\left(X(t)\right)+\sigma\text{d}B_{t}(t),
\end{equation}
 where $\sigma^{2}=\frac{1}{2\beta}$ and $B_1,\dots, B_N$ are independent Brownian motions on $\spin$. Imagine, for example, the initial
configuration in which $X_{i}=\frac{2\pi i}{N}$, so the winding number
in the beginning equals $1$. In order to exit the $\windnum=1$ phase,
we must overcome a certain energy barrier. Indeed, the energy of each
pair in the beginning is $-J\cos\left(\frac{2\pi}{N}\right)\approx-J$,
but the winding number could only change when some pair has an interaction
energy $-J\cos\left(\pi\right)=J$. That is, the energy barrier is
$2J$. Thus, for fixed $N$, we expect a waiting time of the scale
$e^{J/\sigma^{2}}$ to change the winding number as $\beta J\rightarrow\infty$
\cite{BovierDenHollander,OlivieriVares}.

In this paper we will also be interested in the dependence of this
time on $N$, which will add an entropic term to that expression --
since there are $N$ possible positions for neighboring rotors
to point in opposite directions, the entropy of the transition state
should equal approximately $\log N$, giving the time scale $e^{J/\sigma^{2}-\log N}$.
Probabilistically, this is the result we expect if the winding number
changed independently at each pair $\left(X_{i},X_{i+1}\right)$ with
rate $e^{-J/\sigma^{2}}$. We prove that this is indeed the correct
time scale in which the winding number changes.

\section{Model and notation}

For $N\in\nn$ we consider a periodic chain of rotors $X^{(N)}_1,\dots,X^{(N)}_N=X^{(N)}_0$ on the one dimensional torus $\spin=\rr/2\pi\zz$. They evolve according to the It\^o equation
\begin{equation} \label{eq:ito}
\dd X^{(N)}(t)  =-\grad H_N\left(X^{(N)}(t)\right)+\sigma_{N}\text{d}B^{(N)}(t),
\end{equation}
with the Hamiltonian
\begin{equation} \label{eq:hamiltonian}
H_N(x_1,\dots,x_N) = -J_N \sum \limits_{i=1}^N \cos(x_i - x_{i-1}),
\end{equation}
where $x_0=x_N$, and where $B^{(N)}=(B_1^{(N)},\dots B_N^{(N)})$ is a vector of independent standard Brownian motions on $\spin$ (each of which obtained through projection of a real-valued standard Brownian motion). In the following we may omit the subscript or superscript $N$ to simplify the notation. We use $C$ to denote a generic positive constant.

This process is reversible with respect to the equilibrium measure $\mu_N$ on $\spin^N$ given by 
\begin{equation}\label{eq:mu}
\dd \mu_N(x_1,\dots,x_N)=\frac{e^{-\frac{1}{2\sigma^2} H_N(x_1,\dots,x_N)}}{\mathcal Z_N}\dd \lambda(x_1,\dots,x_N),
\end{equation}
where $\lambda$ is the Lebesgue measure and $\mathcal{Z}_N$ a normalization constant.
Probabilities and expected values with respect to the It\^o process will be denoted $\pp^{(N)},\ee^{(N)}$ (or simply $\pp,\ee$ when clear from the context). When starting from equilibrium we write $\pp_\mu, \ee_\mu$, and when the initial state is drawn from equilibrium conditioned on being in some set $A$ we write $\pp_{\mu|A}, \ee_{\mu|A}$. We will often consider the process starting from equilibrium conditioned on the event $\{\windnum = k\}$ (recall equation \eqref{eq:windnum}); in this case we use the notation $\pp_{\mu|k}, \ee_{\mu|k}$.
Finally, we denote the hitting time of an event $A\subseteq \spin^N$ by $\tau_{A}$, and for the
event $\left\{ \windnum\neq k\right\} $ (understood to contain also the points in which $\windnum$ is not defined) we write $\tau_{k\pm1}$.

We will start by studying the equilibrium measure $\mu$. This analysis will help us understand the equilibrium properties of $\windnum$, and to obtain estimates that will be used in order to study the dynamics.
%As a byproduct of the above proof, we obtain the asymptotic behavior of the winding number under the stationary measure
We will see that when the temperature is very low, $J/\sigma^2 \gg N$, the winding number equals $0$ with very high probability. On the other hand, in the regime $J/\sigma^2 \ll N$ it satisfies a central limit theorem, fluctuating around $0$. The emergence of non-trivial topology, appearing at temperatures above $N^{-1}$, corresponds in the two dimensional model to temperatures above the Kosterlitz-Thouless critical point \cite{KT73}. 
\begin{thm}\label{thm:cltforw}
Assume that $J/\sigma^2 \rightarrow \infty$ as $N\to\infty$. 
\begin{itemize}
\item If $J/\sigma^2 \ll N$, then
\begin{equation*}
2\pi \sqrt{\frac{J}{2 \sigma^2 N}} \, \windnum^{(N)} \overset{(d)}{\longrightarrow} \mathcal N(0,1).
\end{equation*}
\item If  $J/\sigma^2 \gg N$, then 
\begin{equation*}
\mu(\windnum^{(N)}=0)\rightarrow 1.
\end{equation*}
\end{itemize}
\end{thm}
The proof of this theorem is given at the end of Section \ref{sec:afewprop}.

After establishing these properties about the equilibrium measure, we will bound $\tau_{k\pm 1}$ from below. We follow the ideas of \cite{OlivieriVares}, but some extra care is needed in order to obtain the correct dependence on $N$.
\begin{thm} \label{th:mainTheorem}
Assume there exists a positive constant $C$ such that $C\log N \leq J/\sigma^2 =o(N) $ and that $J = o(N^\alpha)$ for all $\alpha >0$. 
Then for all $\varepsilon>0$ small enough, as $N\rightarrow \infty$,
\begin{equation} \label{eq:lowerBound}
\pp_{\mu|k}\left(\tau_{k\pm1}\le e^{\frac{J}{\sigma^{2}}\,(1-\varepsilon)-\log N} \right) \to 0.
\end{equation}
\end{thm}

We will finish by proving an upper bound for $\tau_{k\pm 1}$ using a variational principle introduced in \cite{AdP} together with a path argument.
\begin{thm} \label{th:mainTheorem2}
%Assume that there exists a positive constant $C$ such that $J/\sigma^2 \geq C \log N$ and that $\sigma^2 = O(N J/\sigma^2)$ as $N\to\infty$.
Assume that there exist positive constants $C,C_1$ such that $C \log N < J/\sigma^2 =o(N)$, and in addition  $\sigma =o(N^\alpha)$ for all $\alpha >0$.
Then for all $\varepsilon>0$ small enough, as $N\rightarrow \infty$,
\begin{equation} \label{eq:upperBound}
\pp_{\mu|k}\left(\tau_{k\pm1}\ge e^{\frac{J}{\sigma^{2}}\,(1+\varepsilon)-\log N} \right) \to 0.
\end{equation}
\end{thm}

\section{A few properties of equilibrium} \label{sec:afewprop}
In this section we will prove some properties of the equilibrium measure $\mu$, that we will use later in the analysis of the dynamics. 

First, we will see that under $\mu$ the chain of rotors can be seen as a random walk conditioned on returning to the origin, with increments in $\mathcal{S}$ that distribute according to the density
\begin{equation}
h(x) = \frac{1}{m} e^{\frac{J}{2\sigma^2} cos(x)},
\end{equation}
where $m=\int e^{\frac{J}{2\sigma^2} cos(x)} \text{d}x$.
\begin{definition} \label{def:SkXi}
Let $(\xi_i)_{i=1}^n$ be a sequence of i.i.d.\  random variables on $\spin$ with density $h$ with respect to the Lebesgue measure, $S_k = \sum_{i=1}^k \xi_i$, and $\Xi$ be a random variable with the uniform distribution on $\spin$ and independent of the $\xi_i$'s. We denote by $\mathrm P$ the the measure of $\Xi$ and the $\xi_i$'s.
\end{definition}

\begin{proposition} \label{prop:RWrepr}
The vector $(\Xi, \Xi+S_1,\dots,\Xi+S_N)$, conditionally on $S_N=0$, has the same law as $(X_0,\dots,X_N)$ under $\mu$. Note that in particular the first and last entry are equal in both vectors.
\end{proposition}
\begin{proof}
Let $h^{\star N}$ denotes the $N$-th iterated convolution of $h$ with itself. For $(x_1,\dots,x_N)\in \spin^{N}$,
\begin{equation*}
\mathrm P\left(\Xi \in \dd x_1, \Xi+S_1 \in \dd x_2,\dots,\Xi+S_{N-1}\in \dd x_N \middle| S_N = 0 \right) = \frac{h(x_1-x_N)\prod_{i=1}^{N-1} h(x_{i+1}-x_i)}{2\pi h^{\star N}(0)},
\end{equation*}
while 
\[  \frac{\dd \mu}{\dd \lambda} (x) =
\frac{h(x_1-x_N)\prod_{i=1}^{N-1} h(x_{i+1}-x_i)}{\mZ_N} m^N.
\]
By integrating that $\mZ_N = 2\pi h^{\star N}(0) m^N$, from which we obtain the statement by equality of the last two displays.
\end{proof}
\begin{corollary} \label{cor:RWrepr}
The vector $(X_1-X_0,\dots,X_{N}-X_{N-1})$ under $\mu$ has the same law as $(\xi_1,\dots,\allowbreak \xi_N)$, conditionally on $\sum_{i=1}^N \xi_i=0$.
\end{corollary}

\begin{corollary}\label{cor:WindNumber} Let $\rho_i$ be the representative of $\xi_i$ in $(-\pi,\pi)$ (i.e. $\rho_i = \widetilde{\xi_i}$)  thought of as a random variable in $\rr$. Let $g$ be the density of $\rho_i/2\pi$ with respect to the Lebesgue measure on $\mathbb{R}$. Then, for all $k\in \{-N,\dots,N\}$,
\begin{equation} \label{eq:probWegalk}
\mu\left(\windnum = k\right) =  \frac{g^{\star N}(k)}{ \sum_{k\in\mathbb Z} g^{\star N}(k)}.
\end{equation}
\end{corollary}
\begin{proof}
By Corollary \eqref{cor:RWrepr}, $\left(X_1-X_0,\dots,X_{N}-X_{N-1}\right)$ under $\mu$ has the law of $(\xi_1,\dots,\xi_{N})$ under the condition that $S_N=0$, so the result follows from the fact that $\windnum =\frac{1}{2\pi} \sum_{i=1}^N \widetilde{X_{i}-X_{i-1}} \eqlaw \frac{1}{2\pi}\sum_{i=1}^N \rho_i$ and the observation that 
\[
\{S_N = 0\}= \left\{\sum_{i=1}^N \rho_i \in 2\pi \mathbb{Z}\right\}.
\]
\end{proof}

The following moments estimate will prove useful:
\begin{proposition} \label{prop:msbeta}
As $\frac{J}{\sigma^2} \rightarrow \infty,$
\begin{align}
m & \sim 2 \sqrt{\pi} \, e^{\frac{J}{2\sigma^2}}\left(\frac{\sigma^2}{J}\right)^{1/2}, \label{eq:m}\\
s^2 \sim  2(2\pi)^{-2} & \left(\frac{\sigma^2}{J}\right),  \qquad \beta^3  \sim \frac{8(2\pi)^{-3}}{\sqrt \pi} \left(\frac{\sigma^2}{J}\right)^{3/2}. \label{eq:s}
\end{align}
where $s^2=\int y^2 g(y) \mathrm{d}y$ and $\beta^3 = \int |y|^3 g(y) \mathrm{d}y$.
\end{proposition}
\begin{proof}
This is a direct application of Laplace's method. See, e.g., estimate (3) in \cite[p.\ 37]{E56}.
\end{proof}

\begin{proposition} \label{prop:estimateProbWeqk} Assume that $J/\sigma^2 \to\infty$ with $J/\sigma^2 =o(N)$. For all $k\in\zz$, there exist two positive constants $c,C$, such that for $N$ large enough and all $M\geq N/2$, for all $y\in [k-1,k+1]$,
\begin{equation} \label{eq:estimateGstarM}
c \frac{\sqrt J}{\sigma} M^{-1/2} \le g^{\star M}(y) \le C \frac{\sqrt J}{\sigma} M^{-1/2}.
\end{equation}
and moreover,
\begin{equation} \label{eq:LBonProbW}
c \frac{\sqrt J}{\sigma} N^{-1/2} \le \mu(\windnum=k) \le C \frac{\sqrt J}{\sigma} N^{-1/2}.
\end{equation}
\end{proposition}

\begin{proof} By equation \eqref{eq:probWegalk}, we see that it is enough to have a good control on the asymptotic behavior of $g^{\star M}$, which could be obtained via a local limit theorem for densities. However, we want the constants $c,C$ not to depend on $J,\sigma$, so we should beware of the error term in the local limit approximation. We therefore rely on the result of \cite{uniformllt} (see also \cite{KoZh98} and references therein). Letting $s,\beta$ are as in Proposition \ref{prop:msbeta} and $A=\sup g=\frac{2\pi}{m}e^{\frac{J}{2\sigma^2}}$, we obtain that there exists a universal constant $C>0$, such that for all $y\in \mathbb R$, 
\begin{align}\label{eq:LLT}
\left|g^{\star M}(y) - \frac{1}{s \sqrt M} \varphi\left(\frac{y}{s \sqrt M}\right) \right| \leq \frac{C}{s^5 M} \frac{ \beta^9 \max(1,A^5)}{1+\frac{y^2}{s^2 M}} \leq \frac{C}{s M} \frac{1}{1+\frac{y^2}{s^2 M}},
 \end{align}
where the second inequality follows from proposition \ref{prop:msbeta}.

Since $s^2N\to\infty$, estimate  \eqref{eq:estimateGstarM} is a direct consequence of equation \eqref{eq:LLT}. Moreover, setting $M=N$, the RHS of equation \eqref{eq:LLT} is summable in $y=k\in\mathbb Z$ with vanishing sum when $N\to\infty$.
Furthermore $\frac{1}{s \sqrt N} \sum_{k\in\mathbb Z} \varphi\left(\frac{k}{s \sqrt N}\right)\to \int \varphi = 1$, therefore $\sum_{k\in\mathbb Z} g^{\star N}(k) \to 1$ and so equation \eqref{eq:LBonProbW} follows from Corollary \ref{cor:WindNumber}.
\end{proof}

We can use the tools from the proof of the previous proposition to show Theorem \ref{thm:cltforw}:
\begin{proof}[Proof of Theorem \ref{thm:cltforw}]
We begin with the case $J/\sigma^2 =o(N)$. Relying on equation \eqref{eq:probWegalk}, we see from equation \eqref{eq:LLT} and $\sum_{k\in\zz}g^\star (k) \rightarrow 1$ (cf. proof of Proposition \ref{prop:estimateProbWeqk}) that the following local limit estimate for $\windnum$ holds:
\begin{equation} \label{eq:LLTforWind}
\mu (\windnum = k) = (1+\delta_N) \frac{1}{s\sqrt N}\varphi\left(\frac{k}{s\sqrt N}\right)+\alpha_k^N,
\end{equation}
where $\delta_N \to 0$ and $\left| a_k^N \right| \le C (sN)^{-1} \left( 1+\frac{k^2}{s^2 N} \right)^{-1}$. From the local limit estimate for $\windnum$, we obtain via integration over a test function and the Riemann sum properties that
\begin{equation*} (s^2 N)^{-1/2} \, \windnum^{(N)} \overset{(d)}{\longrightarrow} \mathcal N(0,1).
\end{equation*}

%If, on the other hand, $N\sigma^2/ J \to 0$, observe that since the RHS of \eqref{eq:LLT} is still summable for $y=k\in \mathbb{Z}\setminus\{0\}$ with vanishing sum and that
%\[\sum_{k\neq 0} \frac{1}{s\sqrt N} \varphi\left(\frac{k}{s\sqrt N}\right) \leq \frac{C}{s\sqrt N} \sum_{k\neq 0} e^{-k/(s^2N)} = \frac{2C}{s\sqrt N} \frac{e^{-1/(s^2 N)}}{1-e^{-1/(s^2 N)}} \to 0,\]
%as $N\to \infty$, we have $\sum_{k\neq 0} g^{\star N}(k) \to 0$. Moreover by estimate (1) in \cite{KoZh98},
%\[
%\left|g^{\star N}(0) s\sqrt N - \varphi(0) \right| \leq C \frac{\beta^3}{\sqrt N} \max(1,A^3) \leq \frac{C}{\sqrt N},
%\]
%so that $g^{\star N}(0)\to\infty$. We conclude with \eqref{eq:probWegalk} that $\mu(\windnum = 0)\to 1$.
For the case $N=o(J/\sigma^2)$, we will use the following lemma:
\begin{lem}\label{lem:convolution}
Let $f_1$ and $f_2$ be two even functions from $\rr$ to $\rr$, non-decreasing on $(-\infty,0]$ and non-increasing on $[0,\infty)$. Then their convolution $f_1 \star f_2$ is also even, non-decreasing on $(-\infty,0]$ and non-increasing on $[0,\infty)$.
\end{lem}
Before proving the lemma, we will use it in order to deduce the second part of the theorem. Consider a random variable $Z_N$ whose law is given by $g^{\star N}$. Its variance is $N s_N^2$, so by equation \eqref{eq:s} it converges to $0$. Hence the probability that $Z_N$ is in $(-1/2,1/2)$ is equal $1-\varepsilon_N$ for a sequence $\varepsilon_N\to 0$. Since by Lemma \ref{lem:convolution} the function $g^{\star N}$ is maximal at $0$, necessarily $g^{\star N}(0) > 1-\varepsilon_N$. On the other hand, again using Lemma \ref{lem:convolution}, for all $x>1/2$ 
\[
\sum_{k=1}^N g^{\star N}(k)\One_{(k-1/2,k)}(x)<g(x),
\]
and integrating both sides yields $\sum_{k=1}^N g^{\star N}(k) < 2 \int_{1/2}^\infty g(x)\dd x=\varepsilon_N$. The same estimate holds when summing over negative values of $k$, so $\sum_{k\neq 0}g^{\star N}(k) \to 0$. This concludes the proof by Corollary \ref{cor:WindNumber}.
\end{proof}

\begin{proof} [Proof of Lemma \ref{lem:convolution}]
Convolution of even functions is even, so it suffices to prove that, for any fixed $0<x<y$,
\begin{equation}\label{eq:convolution_decreasing}
f_1 \star f_2 (x) - f_1 \star f_2 (y) = \int_\rr (f_1(x-t)-f_1(y-t))f_2(t)\dd t \ge 0. 
\end{equation}
We will split the integral over $\rr$ into the two intervals $(-\infty,\frac{x+y}{2}]$ and $[\frac{x+y}{2},\infty)$. By the change of variables $s=x+y-t$ the second could be rewritten as
\begin{equation}
\int \limits_{[\frac{x+y}{2},\infty)} (f_1(x-t)-f_1(y-t))f_2(t)\dd t = \int\limits_{(-\infty,\frac{x+y}{2}]} (f_1(y-s)-f_1(x-s))f_2(x+y-s)\dd s.
\end{equation}
Fix $s<\frac{x+y}{2}$. If $s\ge 0$ then, since $x+y-s>s>0$, we know that $f_2(x+y-s)\le f_2(s)$. The same inequality holds also when $s < 0$ -- in that case, $x+y-s>-s>0$, so indeed $f_2(x+y-s)\le f_2(-s) = f_2 (s)$. Similar reasoning shows that $f_1(y-s)\le f_1(x-s)$ -- when $s<x$ this follows from the fact that $0<x-s<y-s$, and when $s\ge x$ we use $0<s-x<y-s$. To summarize,
\begin{equation}
\forall s \in (-\infty,\frac{x+y}{2}], \qquad f_2(x+y-s)\le f_2(s) \quad \text{ and } \quad f_1(y-s)-f_1(x-s)\le 0.
\end{equation}
We can now prove inequality \eqref{eq:convolution_decreasing}:
\begin{multline*}
f_1 \star f_2 (x) - f_1 \star f_2 (y) = \\
\int \limits_{(-\infty,\frac{x+y}{2}]} (f_1(x-t)-f_1(y-t))f_2(t)\dd t + 
\int\limits_{(-\infty,\frac{x+y}{2}]} (f_1(y-s)-f_1(x-s))f_2(x+y-s)\dd s \\
\ge \int \limits_{(-\infty,\frac{x+y}{2}]} (f_1(x-t)-f_1(y-t))f_2(t)\dd t + \int\limits_{(-\infty,\frac{x+y}{2}]} (f_1(y-s)-f_1(x-s))f_2(s)\dd s = 0.
\end{multline*}

\end{proof}
%\begin{remark}
%\textcolor{red}{We have stated a convergence in law for simplicity, we actually have a LLT (write the thing with the deltas)}
%\end{remark}

\section{Lower bound}
In this section we focus on showing the lower bound in equation \eqref{eq:lowerBound} on the time needed for the winding number to change.  Our main tool is the following lemma, in the spirit of \cite[Chapter 5]{OlivieriVares}. In the following, for $x=(x_1,\dots,x_N)\in \spin^N$, we denote by $\norm{x}_\infty=\max \{ |\tilde{x}_i|: i\in [N] \}$.

\begin{lem} \label{lm:LemmaForlowerBound}
Fix two events $A,B$ on $\spin^{N}$ and $\delta>0$. Let
$A_{\delta}$ be the $\delta$-thickening of $A$ in $B$, that is,
\[
A_{\delta}=\left\{ x:\norm{x-y}_{\infty}<\delta\text{ for some }y\in A\right\} \cap B.
\]
Assume that $\tau_A \le \tau_{B^c}$ $\pp_{\mu|B}$-a.s.\ and that there is a sequence $T=T(N)$ such that
\begin{align} \label{eq:asumpt1TimestepsLemma}
{J/\sigma^2 \leq C(\log T + \log N)\to\infty;} \quad \sigma =O((NT)^\alpha),
\end{align}
for all $\alpha > 0$, as $N\to\infty$.
Let
\[\alpha_N := \sigma^2 T(\log T + \log N)\frac{\mu(A_\delta)}{\mu(B)}.\]
Then, for $\delta$ small enough, 
\[
\pp_{\mu|B}\left(\tau_{A}<T\right) = O(\alpha_N)
\]
as $N\to\infty$.
\end{lem}
\begin{proof}
The strategy is to divide time between $[0,T]$ as follows:
\[\Delta = \frac{\delta^2}{80\sigma^2 (\log T + \log N)},\quad 
t_j=\Delta\,j \quad \text{for } 0\le j\le n=\left\lceil \frac{T}{\Delta}\right\rceil.
\]
In order to control the increment of the $X_i$'s during a timelap $\Delta$, we start by defining a "bad event" $\Omega$ where the driving noise varies too much:
\begin{equation*}
\Omega =\left\{ \max_{0\le j\le n}\,\sup_{t_{j}\le t\le t_{j+1}}\sigma\norm{B(t)-B(t_{j})}_{\infty}\ge \frac{\delta}{4}\right \}
\end{equation*}
On its complement $\Omega^{c}$, for all $j\in\left[0,n\right]$, $t\in\left[t_{j},t_{j+1}\right]$ and $i<N$,

\begin{align*}
\left|X_{i}\left(t\right)-X_{i}\left(t_{j}\right)\right| & \le\int\limits _{t_{j}}^{t}\left|\frac{\partial H}{\partial x_i}(X_s)\right|\text{d}s+\sigma\left|B_{i}(t)-B_i(t_j)\right|\\
 & \le J\,\Delta+ \frac{\delta}{4} \leq \frac{\delta}{2},
\end{align*}
where the last equality comes from assumption \eqref{eq:asumpt1TimestepsLemma} of the lemma and choosing $\delta$ small enough. On the event $\Omega^c \cap\{\tau_A \leq \tau_{B^c};\tau_A< T\}$, the largest $t_j$ strictly less than $\tau_A$ satisfies $X_{t_j}\in A_\delta$ (if $\tau_A = 0$, set $t_j = 0$).
Therefore, as $\Omega$ is independent of the initial condition of $X$ and since we assume that $\tau_A \leq \tau_{B^c}$,
\begin{align*}
\pp_{\mu|B}\left(\tau_{A}<T\right) &\leq \pp(\Omega) + \pp_{\mu|B}\left(\Omega^c \cap \{\tau_{A}<T\}\right),\\
& \leq \pp(\Omega) + \pp_{\mu|B}\left(\cup_{j\leq n} \{X(t_j)\in A_\delta \}\right).
\end{align*}

By the union bound, the reflection principle for Brownian motion on $\mathbb R$ and the standard Gaussian tail estimate $\int_{u}^\infty e^{-z^2/2} \dd z \leq u^{-1} e^{-u^2/2}$,
\begin{align*}
\pp\left(\Omega\right) & \le 2N\left(n+1\right)\frac{4\Delta^{1/2} \sigma}{\delta} e^{-\frac{\delta^{2}}{32\sigma^{2}\Delta}}\\
&\leq C e^{\log N + \log T} \frac{\sigma}{\delta \Delta^{1/2}} e^{-\frac{\delta^{2}}{32\sigma^{2}\Delta}}\\
& = C' e^{-\frac{3}{2}(\log N+\log T)} (\log N + \log T)^{1/2} \frac{\sigma^2}{\delta^2} \to 0,
\end{align*}
where the convergence is assured by \eqref{eq:asumpt1TimestepsLemma}.

Now, by stationarity, for any $t\geq 0$,
\[
\pp_{\mu|B}\left(X_{t}\in A_{\delta}\right)\leq\frac{\pp_{\mu}\left(X_{t}\in A_{\delta}\right)}{{\pp_{\mu}\left(X_{0}\in B\right)}}=\frac{\mu\left(A_{\delta}\right)}{\mu\left(B\right)},
\]
so in particular,
\begin{align*}
\pp_{\mu|B}\left(X_{t_{j}}\in A_{\delta}\text{ for some }j\right) & \le n\,\frac{\mu\left(A_{\delta}\right)}{\mu\left(B\right)} \\
 & \le\frac{80\sigma^{2}\left(\log T+\log N\right)}{\delta^{2}}T\,\frac{\mu\left(A_{\delta}\right)}{\mu\left(B\right)}\\
 & =O(\alpha_N).
\end{align*}
\end{proof}

\noindent\textbf{Conclusion of the proof of Theorem \ref{th:mainTheorem}.}
Let $B=\left\{ \windnum=k\right\} $ and $A$ be the event 
\[
A=\left\{ \exists i\text{ such that }X_{i}-X_{i-1}=\pi\right\} .
\]
We will estimate the probability of its thickening. By Proposition \ref{prop:estimateProbWeqk}, for $N$ large enough and our assumptions on $J/\sigma^2$,
\begin{align*}
&\mu\left(X_{1}-X_{0}\in\left(\pi-\delta,\pi+\delta\right)|\windnum=k\right) 
 = \int_{\left[\frac{\pi-\delta}{2\pi},\frac{1}{2}\right]\cup \left[\frac{-1}{2},\frac{-\pi+\delta}{2\pi}\right]}\frac{g(y)g^{\star (N-1)}(k-y)}{g^{\star N}(k)}\text{d}y\\ &\leq C \int_{\left[\frac{\pi-\delta}{2\pi},\frac{1}{2}\right]\cup \left[\frac{-1}{2},\frac{-\pi+\delta}{2\pi}\right]} g(y)\dd y \leq C\frac{\sqrt{J}}{\sigma}e^{-J/\sigma^{2}\,\left(2-\frac{1}{2}\delta^{2}\right)}.
\end{align*}
Hence, 
\[
\mu(A_{\delta})<C\frac{\sqrt{J}}{\sigma}e^{-\frac J {2\sigma^{2}}\,\left(2-\frac{1}{2}\delta^{2}\right)+\log N}\mu\left(B\right).
\]
Letting
\[
T=e^{\frac{J}{2\sigma^{2}}\,\left(2-\delta^{2}\right)-\log N},
\]
we can estimate $\alpha_N$ of Lemma \ref{lm:LemmaForlowerBound} using our assumptions on $J,\sigma^2$:
\[
\alpha_N \leq \sigma^{2}\,T\left(\log T+\log N\right)\frac{\sqrt{J}}{\sigma}e^{-\frac{J}{2\sigma^{2}}\,\left(2-\frac{1}{2}\delta^{2}\right)+\log N}
=(1-\delta^{2}/2)\,\frac{J^{3/2}}{\sigma}e^{-\frac{\delta^{2}}{4}\,J/\sigma^{2}}\to 0.
\]
Since the winding number changes only when we hit $A$, if
we start in $\windnum=k$ necessarily $\tau_{A}=\tau_{\left\{ \windnum\neq k\right\} }$. We may therefore use Lemma \ref{lm:LemmaForlowerBound} concluding that 
\[
\pp_{\mu|k}\left(\tau_{k\pm 1}<T\right)\to 0.
\]
\qed

\section{Upper bound}
\subsection{Proof of the upper bound for the exit time (Theorem \ref{th:mainTheorem2})}
\label{sec:proofOfprop:upperBTauA}
The upper bound will use certain spectral properties of the infinitesmal generator of the process.
Recall that the generator associated with the process defined in equation \eqref{eq:ito} and its associated Dirichlet form, acting on a smooth function $f$, are given by
\begin{align}
Lf =& -\grad H \grad f + \frac{\sigma^2}{2} \Laplacian f,\\
\mathcal{E} f =& \mathcal{E}(f,f) = \frac{1}{2} \mu\left((\grad f)^2\right).
\end{align}

Fix a closed set $A\subset \spin^N$. The next lemma is a result of \cite{AdP}, which bounds $\tau_A$ from above.
They consider the discrete case of an interacting particle system; but there is no essential difference between this case and the continuous diffusion equation that we are interested in. There are, however, several technical complications that appear, detailed in Appendix \ref{sec:AdPproof}. Among these complications is a certain regularity condition on the set $A$. For the sake of this discussion, it suffices to assume that its complement $A^c$ satisfies the Poincar\'e cone condition, i.e., that from each point $x$ on the boundary of $A^c$ one can find a cone which is entirely contained in the interior of $A$.

Let $C^\infty_c(A^c)$ be the set of smooth function on $\spin^N$ that are compactly supported in $A^c$.
\begin{lem}\label{lem:AdP}
Let $\lambda>0$, and assume that for all $f\in C_c^\infty(A^c)$,
\begin{equation}\label{eq:AdP_ineq}
\mathcal{E}f \ge \lambda \mu(f^2).
\end{equation}
Then $\pp_\mu(\tau_A>t)\le e^{-\lambda t}$ for all $t>0.$
\end{lem}

This lemma will allow us to find the bound on the hitting time that we are looking for.
Starting from winding number $k$, the corresponding set $A$ we are interested in is $\left\{ \windnum\neq k\right\}$. It turns out, thought, that it is more convenient to consider the set 
\begin{equation}\label{eq:A}
A=\left\{ \windnum\neq k\right\} \cup E,
\end{equation}
where $E$ is some "bad" event that we will define below. In fact, we will show that $\tau_E$ is  usually much bigger than $\tau_A$ (see equation \eqref{eq:upperBoundTauA} and Proposition \ref{prop:tauEisBig}), so that $\tau_A$ indeed captures the time needed for the winding number to change.
To obtain the upper bound on $\tau_A$, using Lemma \ref{lem:AdP}, it will be enough to prove inequality \eqref{eq:AdP_ineq}. This is done via a path argument (see, e.g., \cite[Section 13.5]{levinpereswilmer}). By averaging over different paths, as in \cite[Section 13.5.1]{levinpereswilmer}, we will be able to find the correct entropic term.

Start by defining the bad event. Let $r\in \nn$ and $\delta>0$ be two parameters (independent of $N$) to be determined later on, and consider the intervals $rj,\dots,rj+k-1$ for $j=1,\dots,\floor{N/r}$. We say that such an interval is \emph{good} (for a state $x\in \spin^N$) if $|x_{i+1}-x_i|\leq 3\delta$ for all $i\in \{rj,\dots,rj+r-1\}$. The set of $j$ such that $rj,\dots,rj+k-1$ is a good interval is denoted by $G_x$.
Then the \emph{bad event} is
\begin{equation}\label{eq:badevent}
E = \{x : |G_x| < \floor{N/2r} \}.
\end{equation}

\begin{remark}
By our definition of $\windnum$, the set $\{\windnum \neq k\}$ is closed, hence $A$ is closed. Furthermore $A^c$ satisfies the Poincar\'e cone condition, so that by Corollary \ref{cor:poincarecone} and Lemma \ref{th:asselah} we may use Lemma \ref{lem:AdP} with this set $A$.
\end{remark}

The following proposition bounds from below the time needed to enter $E$.
\begin{proposition} \label{prop:tauEisBig}
%Suppose that $\sigma^2 = O(N J/\sigma^2 + \log N)$ and that $J/\sigma^2 \to \infty$.  Then there exist $C,c>0$, such that for $N$ large enough
%\[ \pp_\mu(\tau_{E} < e^{c \frac{J}{\sigma^2}N}) \leq C \sigma \sqrt J\left({c} \frac{J}{\sigma^2} + \log N\right) e^{-{c}\frac{J}{\sigma^2} N}.\]
Suppose that $\frac{J}{\sigma^2} \rightarrow \infty$ with $J/\sigma^2 = o(N)$ and $\sigma = o(N^\alpha)$ for all $\alpha>0$. Then, for $N$ large,
\[
\pp_\mu\left(\tau_{E} < e^{\frac{2J}{\sigma^2}}\right) = o \left(\sqrt{\frac{J}{\sigma^2 N}} \right).
\]
\end{proposition}
We begin with a lemma that we will use in the proof of the proposition.
\begin{lem} \label{lem:comparaison}
Assume that $J/\sigma^2 = o(N)$. Then for every measurable set $F\subseteq \spin^N$,
\[
\mu(F) \leq C \sqrt{\frac{J}{\sigma^2}}\, \mu^{\text{ind}}(F),
\]
where $\mu^{\text{ind}}$ denotes the law of $(\Xi, \Xi+S_1,\dots,\Xi+S_{N-1})$ (recall Definition \ref{def:SkXi}). 
\end{lem}
\begin{proof}
Let $\phi(x_1,\dots,x_N)$ be any non-negative measurable function. By Proposition \ref{prop:RWrepr}, 
%\begin{align*}
%\mu(\phi(X_0,\dots,X_{N-1})) & = \int_{\spin^{N+1}} \phi(y_0,y_1,\dots,y_{N-1}) \frac{f(y_1-y_0) \dots f(y_N-y_{N-1}) f(0-y_N)}{f^{\star N}(0)} \, \frac{\dd y_0}{2\pi} \dd y_1 \dots \dd y_N \\
%& \leq \frac{e^{\frac{J}{2\sigma^2}}}{m f^{\star N}(0)} \int_{\spin^{N+1}} \phi(y_0,y_1,\dots,y_{N-1}) f(y_1-y_0) \dots f(y_N-y_{N-1}) \, \frac{\dd y_0}{2\pi} \dd y_1 \dots \dd y_N \\
%& \leq C \sqrt{\frac{J}{\sigma^2}}\, \mu^{\text{ind}}(\phi(X_0,\dots,X_{N-1})),
%\end{align*}
\begin{align*}
\mu(\phi(X_1,\dots,X_N)) = & \int \frac{h(x_1-x_N)\prod_{i=1}^{N-1} h(x_{i+1}-x_i)}{2\pi h^{\star N}(0)} \phi(x_1,\dots,x_N) \dd x_1 \dots \dd x_N \\
 \leq & \frac{\sup h}{h^{\star N}(0)} \int \frac{\prod_{i=1}^{N-1} h(x_{i+1}-x_i)}{2\pi}\phi(x_1,\dots,x_N) \dd x_1 \dots \dd x_N \\
 \leq & C \sqrt{\frac{J}{\sigma^2}}\, \mu^{\text{ind}}(\phi(X_1,\dots,X_N)),
\end{align*}
where we have used equation \eqref{eq:m} and the fact that $h^{\star N}(0)$ is bounded away from $0$ uniformly in $N$. The latter could be deduced from the fact that $\sum_k g^{\star N}(k)\to 1$ (see proof of Proposition \ref{prop:estimateProbWeqk}), which implies that $h^{\star N}(0) \rightarrow (2\pi)^{-1}$. 
\end{proof}

\begin{proof}[Proof of Proposition \ref{prop:tauEisBig} ]
The proof of the proposition is based on Lemma \ref{lm:LemmaForlowerBound}, with $E$ and $\mathcal{S}^N$ playing the roles of $A$ and $B$ respectively. We thus wish to bound from above $\mu(E_\delta)$, where $E_\delta$ is the $\delta$-thickening of $E$. Note that $E_\delta \subset \left\{\sum_{i=1}^N \mathbf{1}_{|X_i-X_{i-1}| \geq \delta} \geq \frac{N}{2r}\right\}$, and that the probability of the last event with respect to $\mu^{\text{ind}}$ is bounded from above by $e^{-NI(a)}$, with $I$ the rate function of the Bernoulli random variable with parameter $p_\delta = P(|\xi_1| \geq \delta)$  and $a=(2r)^{-1}$. As $\frac{J}{\sigma^2}$ tends to infinity, $p_\delta \rightarrow 0$ (e.g. by Chebyshev's inequality and equation \eqref{eq:s}), so $I(a)\rightarrow \infty$.
We can now use Lemma \ref{lm:LemmaForlowerBound} -- by Lemma \ref{lem:comparaison},
\begin{align*}
\alpha_N = \sigma^2 e^{\frac{2J}{\sigma^2}} \left(\frac{2J}{\sigma^2} + \log N \right)\mu(E_\delta) \le
 C \sigma^2 e^{\frac{2J}{\sigma^2}} \left(\frac{2J}{\sigma^2} + \log N \right)\sqrt{\frac{J}{\sigma^2}} e^{-I(a)N}\\
=o \left(\sqrt{\frac{J}{\sigma^2 N}} \right),
\end{align*}
and this concludes the proof.
%By \eqref{eq:m} and using the bound $\cos y \geq 1-y^2/2$,
%\[p_\delta \leq 2 \int_{\delta}^\pi \frac{e^{-\frac{J}{2\sigma^2} \cos \theta}}{m} \leq C \frac{\sqrt J}{\sigma} e^{-\frac{J}{4\sigma^2}\delta^2} \to 0
%\]
%as $N\to\infty$, since $J/\sigma^2\to\infty$. It is therefore straightforward to check  $I(a)\geq c J/\sigma^2$ as $N\to\infty$, for some $c=c(\varepsilon)>0$, so that by by Lemma \ref{lem:comparaison},
%\[\mu(C_\delta) \leq C \frac{\sqrt J}{\sigma} e^{-c \frac{J}{\sigma^2}N}.\]
%Letting $T=e^{\frac{c}{2} \frac{J}{\sigma^2}N}$,  we conclude via Lemma \ref{lm:LemmaForlowerBound}, whose assumption is verified whenever $\sigma^2 = O(N J/\sigma^2 + \log N)$. 
\end{proof}

\noindent \textbf{Conclusion of the proof of the upper bound (Theorem \ref{th:mainTheorem2}).}
As explained above, we are reduced to studying the entry time of $A$, recalling equations \eqref{eq:A} and \eqref{eq:badevent}.
We postpone to Section \ref{subsec:proofPoincare} the proof of the following lemma:
\begin{lem} \label{lem:variational_ineq} 
For all $\varepsilon >0$, we can choose $\delta$ small enough and $r$ large enough such that for every $f$ in $C^\infty_c(A^c)$, 
\[\mu(f^2) \leq C e^{\frac{J}{\sigma^2} (1+\varepsilon) - \log N} \, \mathcal Ef.\]
\end{lem}
Define
\[T = e^{\frac{J}{\sigma^2} (1+2\varepsilon) - \log N},\]
and observe that since $\tau_A > 0$ implies $\windnum(X_0) = k$, Lemma \ref{lem:variational_ineq}, Lemma \ref{lem:AdP} and equation \eqref{eq:LBonProbW} yield (recall that $J/\sigma^2 > C \log N$):
\begin{equation} \label{eq:upperBoundTauA}
\pp_{\mu|k}(\tau_A > T) = \frac{\pp_\mu\left(\tau_A >T \right)}{\mu(\windnum = k)}
=o(1).
%\leq C\times \exp\left\{-e^{-\varepsilon J/\sigma^2 + \delta J/\sigma^2}  + \frac{1}{2} \log N\right\} \frac{\sigma}{\sqrt J}.
\end{equation}
%which, if $\varepsilon < \delta$, vanishes as $N\to\infty$.
We further obtain from equation \eqref{eq:LBonProbW} that
\begin{align*}
\pp_{\mu|k}(\tau_{k\pm 1} >T) & = \pp_{\mu|k}(\tau_{k \pm 1} > T, \tau_{E}> T) +  \pp_{\mu|k}(\tau_{k\pm 1} > T, \tau_E < T)\\
& \leq \pp_{\mu|k}(\tau_A >T) + C\frac{\sigma}{\sqrt J} N^{1/2}\, \pp_\mu\left(\tau_E < T\right).\\
\end{align*}
The first term converges to $0$ by equation \eqref{eq:upperBoundTauA}, and by Proposition \ref{prop:tauEisBig} also does the second one. \qed

\subsection{Proof of Lemma \ref{lem:variational_ineq}}
\label{subsec:proofPoincare}
We define for each $x\in A^c$ a family of paths starting in $A$ and leading to $x$. In order to obtain the correct entropic term, we need these paths to be disjoint, and we will construct them such that along the $j$-th path only the coordinates $x_{rj},\dots,x_{rj+r-1}$ change (recalling the definition of a good interval). We therefore define $\Pi_{j}$ to be the projection on these coordinates, i.e., 
\[
\Pi_{j}(x_{1},\dots,x_{N})=\left(0,\dots,0,x_{rj},\dots,x_{rj+r-1},0,\dots,0\right).
\]

%In $A^{c}$ there are enough $j$s such that the angles in $rj,\dots,\left(j+1\right)r$
%are small. Let $J_{x}$ \textcolor{blue}{($J$ is the energy. Maybe use another notation ?)} be the set of indices $j$ such that the angles in coordinates $rj,rj+1,\dots,rj+r$ are small.
\begin{proposition}\label{prop:paths}
Fix $x\in A^c$. Then for all $j\in G_x$ there exists a function
$\phi_{j}^{x}:\left[0,2\pi\right]\rightarrow \spin^N$
such that
\begin{enumerate}
\item $\phi_{j}^{x}(0)\in A$ and $\phi_{j}^{x}(2\pi)=x$,
\item $\Pi_{j}\left(\frac{\dd \phi_{j}^{x}\left(t\right)}{\dd t}\right)=\frac{\dd \phi_{j}^{x}\left(t\right)}{\dd t}$ with $\Pi_j$ as above,
\item $\left(\frac{\dd \phi_{j}^{x}\left(t\right)}{\dd t}\right)^{2}\le r$
for all $t$,
\item Let $D\phi_{j}^{x}\left(t\right)$ be the derivative of $\phi$
with respect to $x$, i.e. the matrix whose $\left(\alpha,\beta\right)$ entry
is the $\alpha$ coordinate of the vertor $\frac{\partial\phi_{j}^{x}\left(t\right)}{\partial x_{\beta}}$; 
then $\left|\det\,D\phi_{j}^{x}(t)\right|^{-1}\le1$,
\item $H\left(\phi_{j}^{x}\left(t\right)\right)\leq H(x)+\Delta H$
for all $t$, where $\Delta H=2J+2\pi J\left(\delta+\frac{1}{r}\right)$.
\end{enumerate}
\end{proposition}
\begin{proof}
Denoting $y=\phi_{j}^{x}\left(t\right)$, set
\[
y_{i}=\begin{cases}
x_{i}+\frac{i-rj+1}{r}\,\left(2\pi-t\right) & i\in rj,\dots,rj+r-1\\
x_{i} & \text{otherwise}
\end{cases}.
\]
We need to verify 1-5. It is clear that $y=x$ at time $t=2\pi$. So for condition 1, it is enough to show that $\windnum(\phi_{j}^{x}(0))\neq k$ and we will prove that in fact $\windnum(\phi_{j}^{x}(0))= k+1$. We see that the change in the rotation angle of the path along $rj,\dots,\left(j+1\right)r$ between
times $2\pi$ and $t$ is
\begin{align*}
\sum_{i=rj}^{r(j+1) - 1}\widetilde{y_{i+1} - y_{i}} & =\sum_{i<(j+1)r-1} \widetilde{\left(x_{i+1}-x_{i}+\frac{1}{r}\,(2\pi-t)\right)}\\
 & +\widetilde{\left(x_{\left(j+1\right)r}-x_{\left(j+1\right)r-1} - (2\pi-t )\right)}.
\end{align*}
In a good interval all angles are small, hence $\widetilde{(x_{i+1}-x_{i}+r^{-1}\,\left(2\pi-t\right))}$
is continuous in $t\in[0,2\pi]$ for all $i<\left(j+1\right)r-1$. However, $\widetilde{(x_{(j+1)r}-x_{(j+1)r-1}-(2\pi-t))}$
jumps by $-2\pi$ exactly once from time $0$ to time $2\pi$. This shows
that indeed the winding number has jumped from $k+1$ to $k$ from time $0$ to $2\pi$.

Condition 2 is clear, since we only change the coordinates $rj,\dots,rj+r-1$.

For condition 3, we calculate
\[
\left(\frac{\text{d}\phi_{j,r}^{x}\left(t\right)}{\text{d}t}\right)^{2}=\sum_{i=rj}^{r(j+1)-1}\left(\frac{i-rj+1}{r}\right)^{2}=\frac{1}{r^{2}}\sum_{i=1}^{r}i^{2}\le r.
\]

Condition 4 is satisfied since $D\phi_{j,r}^{x}$ is the
identity matrix.

Condition 5 is verified by calculating, for $i<\left(j+1\right)r-1$,
\begin{equation*}
\begin{aligned}
\cos\left(x_{i+1}-x_{i}\right)-\cos\left(y_{i+1}-y_{i}\right) & =\cos\left(x_{i+1}-x_{i}\right)-\cos\left(x_{i+1}-x_{i}+\frac{1}{r}\,\left(2\pi-t\right)\right)\\
 & \le \sup_{s\in [0,2\pi]} \left|\frac{1}{r}\sin\left(x_{i+1}-x_{i}+\frac{1}{r}\,\left(2\pi-s\right)\right)\right|t\\
 & \le\left(\delta+\frac{1}{r}\right)\frac{2\pi}{r};
\end{aligned}
\end{equation*}
and for $i=\left(j+1\right)r-1$ we simply bound $\left|\text{cos}(\cdot)\right|$ by $1$.
\end{proof}
These paths could be used now in order to prove Lemma \ref{lem:variational_ineq}. We will use the notation $\phi_j^x$ for the functions defined in Proposition \ref{prop:paths} when $j\in G_x$, and when $j \notin G_x$ the function $\phi_j^x(t)$ is defined to be the constant function $x$. Note that, in both cases, conditions 2-5 of Proposition \ref{prop:paths} are satisfied; but not condition 1.
\begin{proof}[Proof of Lemma \ref{lem:variational_ineq}]
Let $\Delta H=2J+2\pi J\left(\delta+\frac{1}{r}\right)$ and let $f$ be any function in $C^\infty_c(A^c)$. For $x\in A^{c}$ and for all $j\in G_x$, noting that $f(\phi_j^x(0))=0$,
\[
f(x) =\int_0^{2\pi}\nabla f\left(\phi_{j}^{x}\left(t\right)\right)\,\frac{\text{d}\phi_{j}^{x}\left(t\right)}{\text{d}t}\dd t,
\]
and thus 
\begin{align*}
f(x)^{2} & =\frac{1}{\left|G_{x}\right|}\sum_{j\in G_{x}}\left(\int_0^{2\pi}\Pi_{j}\grad f\left(\phi_{j}^{x}\left(t\right)\right)\,\frac{\text{d}\phi_{j}^{x}\left(t\right)}{\text{d}t}\dd t\right)^{2}\\
& \le\frac{1}{\left|G_{x}\right|}\sum_{j\in G_{x}}\int_0^{2\pi}\left(\Pi_{j}\grad f\left(\phi_{j}^{x}\left(t\right)\right)\right)^{2}\dd t\,\int_0^{2\pi}\left(\frac{\dd \phi_{j}^{x}\left(t\right)}{\dd t}\right)^{2} \dd t\\
 & \le\frac{2\pi r}{N/2r}\sum_{j=1}^{\floor{N/k}}\int_0^{2\pi}\left(\Pi_{j}\grad f\left(\phi_{j}^{x}\left(t\right)\right)\right)^{2} \dd t,
\end{align*}
where in the last inequality we used $x\in E^c$, and the fact that $\frac{\dd \phi_{j}^{x}\left(t\right)}{\dd t}=0$ if $j\notin G_x$.

We can now integrate over $x$, obtaining 
\begin{align*}
\mu (f^{2}) & =\int\limits_{x\in A^{c}}e^{-H/2\sigma^{2}}f(x)^{2}\dd x \\
 & \le\int\limits _{x\in A^{c}}\text{d}x\,\frac{4\pi r^2}{N}\sum_{j=1}^{\floor{N/k}}\int\limits _{0}^{2\pi}\,e^{-H(x)/2\sigma^{2}}\left(\Pi_{j}\grad f\left(\phi_{j}^{x}\left(t\right)\right)\right)^{2} \dd t\\
 & \le \frac{4\pi r^2}{N} e^{\Delta H/2\sigma^{2}}\int\limits _{x\in A^{c}}\dd x\,\sum_{j=1}^{\floor{N/k}}\int\limits _{0}^{2\pi}\,e^{-H\left(\phi_{j}^{x}\left(t\right)\right)/2\sigma^{2}}\left(\Pi_{j}\grad f\left(\phi_{j}^{x}\left(t\right)\right)\right)^{2}\dd t,
 \end{align*}
so that 
 \begin{align*}
\mu (f^{2}) &  \leq \frac{4\pi r^2}{N} e^{\Delta H/2\sigma^{2}}\int\limits _{y\in A^{c}}\dd y\,\sum_{j=1}^{\floor{N/k}}\int\limits _{0}^{2\pi}\,e^{-H(y)/2\sigma^{2}}\left(\Pi_{j}\grad f\left(y\right)\right)^{2}\left|\det\,D\phi_{j}^{x}(t)\right|^{-1} \dd t\\
 & \le\frac{4\pi r^{2}}{N}e^{\Delta H/2\sigma^{2}}\int\limits _{y\in A^{c}}\text{d}y\,e^{-H(y)/2\sigma^{2}}\left(\grad f\left(y\right)\right)^{2}\\
 & =\frac{8\pi r^{2}}{N}e^{\Delta H/2\sigma^{2}}\mathcal{E}f.
\end{align*}
Choosing $\delta<\varepsilon/4\pi$ and $r > 4\pi/\varepsilon$ finishes the proof of the theorem.
\end{proof}

\section{Concluding remarks and further questions}
This paper presents a certain aspect of metastable topological phases in a simple toy model. We have identified a sharp transition in the behavior of the system's dynamics -- roughly speaking, when $J/\sigma^2>\log N$ different topological phases are metastable, surviving for a long time; while for $J/\sigma^2<\log N$ the winding number changes constantly over a very short time scale. In particular, staring in $\{\windnum=0\}$, the winding number will remain $0$ for a very long time in the regime $J/\sigma^2>\log N$. This is a remarkable behavior, considering that from the equilibrium point of view only when $J/\sigma^2$ reaches the scale $N$ ($\gg \log N$) the state $\windnum=0$ is thermodynamically stable.

This work allows us to gain access to the time scales related to the topological phases of this model, but many questions are left open. 
In view of general results in metastability (see, e.g., \cite{BovierDenHollander,OlivieriVares}), we expect some form of loss of
memory over short time scales, leading to the following conjecture:
\begin{conjecture}
\label{conj:main_conjecture}Assume that $\liminf \frac{J}{\sigma^2 \log N} > 1$,
and let
\[\tau^{(N)}=\frac{1}{2}\, \inf\left\{ t:\pp^{(N)}_{\mu|0}\left(\tau_{0\pm1}>t\right)<e^{-1}\right\}.
\]
Then:
\begin{enumerate}[label=(\roman*)]
\item $\tau^{(N)}=C(J_{N},\sigma_{N})e^{J_{N}/\sigma_{N}^{2}-\log N},$ where
$C(J,\sigma)$ is polynomial in $J,\sigma$.
\item Let $S>0$, and $n_{s}^{(N)}=\windnum\left(X^{(N)}({s/\tau^{(N)}})\right)$
for $s\in\left[0,S\right]$. Then, as $N\rightarrow\infty$, the process
$\left(n_{s}\right)_{0\le s\le S}$ converges to a continuous time
simple random walk, with exponential waiting times of rate $1$.
\end{enumerate}
\end{conjecture}
A natural strategy for proving that waiting times are exponential, as in \cite{OlivieriVares}, is showing a type of metastable mixing. One way to do this is coupling the process starting at two different points with the same winding number, say, for simplicity, $\windnum=0$. It is reasonable to assume that both processes would typically spend their time in the box $[-s,s]^N$, so they could be coupled in a time scale of $\sigma^2/s^2$. As long as this scale is much smaller than the exit time $\tau_{0\pm 1}$ we should observe a loss of memory resulting in the random walk picture laid out in Conjecture \ref{conj:main_conjecture}. 
This strategy, however, is not straight forward to implement when $N$ is big.
A first problem is that if we want to be able to consider small values of $J/\sigma^2$ (i.e. $(1+\varepsilon)\log N$), the processes will spend some of their time in a region in which $H$ is not convex, repelling one another. If we take stronger interaction (say $J/\sigma^2 > 100 \log N$), the time it takes to exit the region in which $H$ is convex is long, and there is hope that a coupling argument will work.
There is, however, a second difficulty that arises -- even though we know that the exit time from this "good" region is long starting from a random configuration (according to $\mu(\cdot|\{\windnum=0\})$), coupling requires an a priori bound on this exit time when starting from a fixed, deterministic, configuration. Such an estimate is, for the moment, beyond our reach. 
Figure \ref{fig:conjecture}
shows a simulation of the process demonstrating the random walk picture
and the exponential jump times.
\begin{figure}
\begin{tabular}{cc}
\includegraphics[scale=0.3]{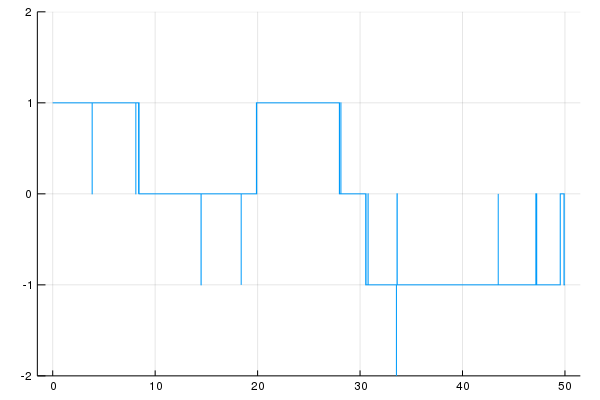} & \includegraphics[scale=0.3]{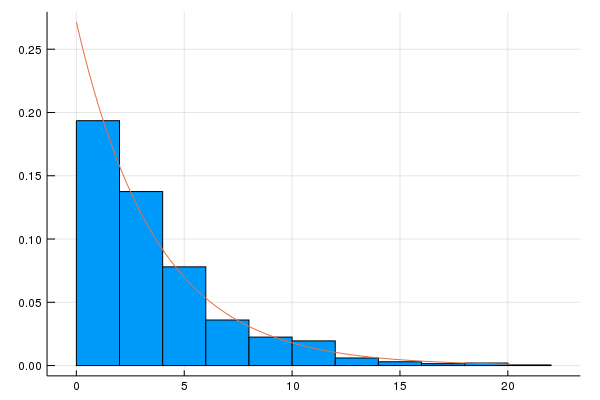}\tabularnewline
\end{tabular}

\caption{\label{fig:conjecture}Left: a simulation of the winding number for
$N=100,J=20,\sigma=3$ as a function of time. Right: a histogram of
exit times from the phase $\protect\windnum=1$, compared to an exponential
distribution, for $N=20,J=20,\sigma=3$. The code used to produce
these graphs is attached to this paper.}

\end{figure}

It is also possible to consider times beyond Conjecture \ref{conj:main_conjecture} -- we have conjectured that the random walk is symmetric due to the fact that the phases with different winding numbers look the same, as long as these winding numbers remain small (i.e., $O(1)$ as $N\rightarrow\infty$). In particular, if we give each phase $\{\windnum=k\}$ an effective energy $H(k)=\inf_{x\in\{\windnum(x) = k\}} H(x)$, these energies will differ by $O(k^2/N)$, making the effective energy landscape essentially flat. However, at longer time scales the small drift induced by these energy differences becomes significant, and when we reach times of order $N\tau^{(N)}$ the process should converge to an Ornstein-Uhlenbeck process. This conjecture is supported by Theorem \ref{thm:cltforw}, showing that $\windnum$ indeed satisfies a central limit theorem in equilibrium.

Another interesting aspect of this model that we have not at all considered in this paper is the behavior of the system
in the metastable phase. Consider, for example, the correlation $\cos\left(X_{1}-X_{\floor{N/2}}\right)$
for large $J$. When the winding number is $0$ all rotors align, and
we expect a positive correlation. However, in the metastable phase
with $\windnum=1$ this is no longer the case -- if we believe that
the typical behavior is when the energy is minimal, $X_{1}$ and $X_{\floor{N/2}}$
should point in opposite directions, and the correlation would be
negative. This effect is indeed observed in simulations, comparing
the $\windnum=1$ phase, the $\windnum=0$ phase, and the system with
free boundary conditions evolving from the same initial configurations.
The complete code is available as a supplementary file to this paper.

One more observable that could be studied in the metastable state
is the response to an external field, i.e., replacing the Hamiltonian
in equation \eqref{eq:hamiltonian} by 
\[
H(x_{1},\dots,x_{N})=-J\sum\cos(x_i - x_{i-1})-B\sum\cos(x_i).
\]
In the $\windnum=0$ phase, where all rotors are aligned, adding a
field will make all rotors point in the direction of the field. In
the $\windnum=1$ phase, however, the rotors cannot all point in the
same direction (e.g., $X_{1}$ and $X_{\floor{N/2}}$ are inverse).
This effect does not allow the rotors to align all together with the
field, so we expect a smaller total magnetization. This behavior is
also verified in simulations.

Metastability effects related to topological phases could be studied
in more models. In one dimension, other interactions could be studied,
as well as models with different topologies (e.g., rotors that live
on an 8 shape). Finally, understanding the dynamics and the metastability
properties of models in higher dimensions, and in particular the creation
and annihilation of vortex pairs in the two dimensional XY model and the equivalent of the regime $\log N<J/\sigma^2<N$ above the Kosterlitz-Touless critical temperature,
is a very interesting (and very challenging) problem.

\appendix

\section{Proof of Lemma \ref{lem:AdP}} \label{sec:AdPproof}

Let $A\subset \spin^N$  be a closed set. For continuous functions in $H_A = \{f\in L^2(\mu), f\equiv 0 \text{ on } A\}$ we let $(P_t)_{t\geq 0}$ be the semi-group acting through $P_t(f)(x) = \ee_x[f(X(t\wedge\tau_A))]$, whose action can be further extended on $H_A$ (see \eqref{eq:contraction}). We emphasize that if $f$ vanishes on $A$, also $\ee_x[f(X(t\wedge\tau_A))]$ vanishes on $A$. We denote by $\Vert f \Vert^2_{H_A} = \mu(f^2\mathbf{1}_{A^c})$ the induced norm in $H_A$ and $\langle , \rangle_\mu$ its associated scalar product.

We start with this lemma:
\begin{lem} \label{lem:operator}
The operator $P_t$ is a strongly continuous symmetric contraction on $H_A$ and if $L^A$ denotes its negative definite self-adjoint generator, then $P_t = e^{tL^A}$ where the exponential is defined via the spectral theorem. Moreover, the domain of $L^A$ contains $C^\infty_c(A^c)$ and for all $f\in C^\infty_c(A^c)$, $L^A f = L f$.
\end{lem}
\begin{proof}
We now verify that $(P_t)$ is a strongly continuous symmetric contraction semi-group on $H_A$. For the contraction property, observe that by Jensen's inequality,
\begin{equation} \label{eq:contraction}
\Vert  P_t f\Vert^2_{H_A} = \int_{A^c} \mu(x) \ee_x\left[f\left(X(t) \right)\mathbf{1}_{t<\tau_A}\right]^2 \dd x \leq \int_{\spin^N} \mu(x) \ee_x\left[f\left(X(t)\right)^2\right]  = \Vert f \Vert_{H_A}^2.
\end{equation}
where the last equality comes from the fact that $\mu$ is a reversible measure of $X$. To prove symmetry, one can for example follow the lines of \cite[p.35]{CZ95} and use again that $\mu$ is a reversible measure of $X$.
To show strong continuity, let $f$ be a continuous and bounded function in $H_A$, so $f(X(t\wedge\tau_A))$ is c\`adl\`ag and therefore $P_t f$ converges pointwise to $f$ as $t\to 0$. By dominated convergence, this implies that $P_t f\to f$ in $H_A$, which can be further extended to any $f\in H_A$ by density of continuous and bounded functions in $H_A$ and the contraction property of $P_t$. The exponential formula follows from \cite[Lemma 1.3.2.]{FuOsTa11}.

To prove the last claim of the lemma, we use Dynkin's formula, from which we obtain that for all $f\in C^\infty_c(A^c)$,
\begin{equation}\label{eq:martingaleEq}
P_t f(x) -f(x) = \ee_x\left[\int_0^t L f (X(s\wedge \tau_A))\dd s\right],
\end{equation}
and hence $L^A f = L f$ on $C^\infty_c(A^c)$.
\end{proof}

We now introduce the symmetric positive definite bilinear form $\mathcal E^A$ associated to $P_t$, which satisfies $\mathcal E^A(f,g) = \langle -L^A f,g\rangle_\mu$ for all $f$ in the domain of $L^A$ and $g$ in the domain of $\mathcal E^A$.
\begin{lem} \label{lem:eigenvaluevalue} Assume that $A$ has a smooth boundary and denote by $\lambda_0$ the smallest eigenvalue of $-L^A$. Then,
\begin{equation} \label{eq:eigenvaluevalue}
\lambda_0 = \inf_{f\in C^\infty_c(A^c),f\neq 0} \frac{\langle -L^A f, f\rangle_\mu}{\langle f,f\rangle_\mu} = \inf_{f\in C^\infty_c(A^c),f\neq 0} \frac{\mathcal E^A(f,f)}{\langle f,f\rangle_\mu}.
\end{equation}
\end{lem}
\begin{proof}
First observe that \eqref{eq:eigenvaluevalue} holds if $C^\infty_c(A^c)$ is replaced by the domain of $L^A$. We now explain why it is possible to exchange them.
 
 For $\alpha>0$, let $H_0^1(A^c)$ denote the closure of $C^\infty_c(A^c)$ for the norm induced by $\mathcal E^A_\alpha = \mathcal E^A + \alpha\langle,\rangle$. Note that the closure does not depend on $\alpha$. We also let
\[G_\alpha = \int_0^\infty e^{-t\alpha} P_t,\]
be the strongly continuous resolvent associated to $P$ and $G'_\alpha$ be the Green function associated to the equation $(L-\alpha)u = f$ with Dirichlet boundary conditions, i.e. the bounded operator in $H_A$ such that for all function $f\in H_A$, $G'_\alpha f \in H_0^1(A^c)$ is the unique solution (in the weak sense) of $(L-\alpha)u = f$. Now let $f\in C^\infty_c(A^c)$. It is a well known fact that since $A$ has a smooth boundary, then $G'_\alpha f \in C^\infty_c(A^c)$. Moreover, we have $G'_\alpha f = G_\alpha f$ by \cite[Theorem 7.15]{BovierDenHollander}\footnote{Note that uniqueness of the martingale problem comes from boundedness of $\grad H$.}. Therefore, the two bounded operators $G_\alpha$ and $G'_\alpha$ both agree on the dense subset $C^\infty_c(A^c)$, hence $G_\alpha=G'_\alpha$ on $H_A$. In particular, this implies that $G_\alpha(H_A)\subset H_0^1(A^c)$.

 Let $D(\mathcal E^A)$ denote the domain of $\mathcal E^A$. As $G_\alpha(H_A)$ is dense in $D(\mathcal E^A)$ for $\mathcal E^A_\alpha$, we finally obtain that $D(\mathcal E^A) = H_0^1(A^c)$. Equation \eqref{eq:eigenvaluevalue} thus follows from $C^\infty_c(A^c)$ being dense in $H_0^1(A^c)$ for $\mathcal E^A_\alpha$ and the inclusion of the domain of $A$ in $D(\mathcal E^A)$.
\end{proof}

\begin{definition} \label{def:regularSet} Let $\tau_B^Y$ denote the hitting time of $B$ for a continuous process $(Y_t)$.
Given two Borel sets $E$ and $F$, a set  $E$ is said to be regular for $F$ with respect to $Y$ if for every $x\in E$, $\pp_x(\tau^Y_F =0 ) = 1$.
\end{definition}

Lemma \ref{lem:AdP} is an immediate corollary of the following proposition. 
\begin{proposition} \label{th:asselah} Let $A$ be a closed domain of $\spin^N$ such that the boundary of $A$ is regular for the interior of $A$ (denoted by $\mathring A$) with respect to $X$. Then 
\begin{equation} \label{eq:asselahBound}
\pp_\mu(\tau_A > t) \leq e^{-\lambda_0(A) t},
\end{equation}
where $\lambda_0(A)$ is given in equation \eqref{eq:eigenvaluevalue}.
\end{proposition}
\begin{proof}
We can in fact assume that $A$ has smooth boundary. If not, one can always approximate $A$ with an increasing sequence $A_n$ of closed subset with smooth boundary and observe that: (i) $\lambda_0(A_n)$ in \eqref{eq:eigenvaluevalue} is non-increasing in $n$; (ii) $\pp_\mu(\tau_{A_n} > t) \to \pp_\mu(\tau_{\mathring A} > t)$ by continuity of $X$; (iii) the hypothesis of regularity of the boundary and the strong Markov property for $X$ imply that a.s. $\tau_{\mathring A} = \tau_{A}$.

Then, by the spectral theorem, Lemma \ref{lem:eigenvaluevalue} and Cauchy-Schwarz inequality, we have that for all $f\in C^\infty_c(A^c)$,
\[\int \mu(\dd x) P_t f(x)  = \langle e^{tL^A} f,\mathbf{1}_{A^c}\rangle_\mu = \int_{\lambda_0(A_n)}^\infty e^{-tu}\, \dd \langle E_u^{-L^A} f,\mathbf{1}_{A^c}\rangle_\mu \leq e^{-t\lambda_0(A_n)} \Vert f \Vert_{H_A}, \]
where $E^{-L^A}$ is the spectral measure associated to $-L^A$. 
Now, let $f_n\in\mathcal C^\infty_c(A^c)$ be a sequence of functions converging to $\mathbf{1}_{A^c}$ in $L^2(\mu)$. Since $P_t$ is a contraction, we have that $P_t f_n \to P_t \mathbf{1}_{A^c}$ in $L^2(\mu)$, which along with the previous bound proves the proposition since $\pp_\mu(\tau_A>t)=\ee_\mu\left(\One_{A^c}\left(X_{t\wedge \tau_A}\right)\right)$.
\end{proof}

\begin{lem} \label{lem:regularBoundary}
If $E$ is regular for $F$ with respect to Brownian motion, then $E$ is regular for $F$ with respect to the diffusion $X$ given in equation \eqref{eq:diffusion}.
\end{lem}
\begin{proof}
Let $\varepsilon >0$ and $x\in E$. By Girsanov's theorem, $\sigma B$ has the law of $X$ on $[0,\varepsilon]$ under the the change of measure on $\pp$ with density $D_\varepsilon = \exp\{\int_0^\varepsilon H(B_s) \dd B_s - \frac{1}{2} \int_0^\varepsilon H(B_s)^2 \dd s\}$.
Therefore, since $H$ is bounded we obtain by Cauchy-Schwarz inequality that there is some $C_0> 0$ such that for all $\varepsilon < 1$,
\[\pp_x(\tau^X_F > \varepsilon) \leq \pp_x(\tau^{\sigma B}_F > \varepsilon)^{1/2} \ee_x[D_\varepsilon^2]^{1/2} \leq C_0 \pp_x(\tau^B_F > \sigma^{2} \varepsilon)^{1/2} \to C_0 \pp_x(\tau^B_F > 0)^{1/2} = 0,
\]
as $\varepsilon \to 0$, where we have used the Brownian motion scaling property $(\sigma B_t) \eqlaw (B_{\sigma^2t})$.  Hence $\pp(\tau^X_F = 0) = 1$.
\end{proof}

\begin{corollary}\label{cor:poincarecone}
If $A^c$ satisfies the Poincar\'e cone condition (see the beginning of Section \ref{sec:proofOfprop:upperBTauA} for a definition), then $A$ satisfies the assumption of Lemma \ref{th:asselah}.
\end{corollary}
\begin{proof}
The proof of \cite[Proposition II.1.13]{Bass95} shows that started from the extremity of a finite size cone, the hitting time of the interior of the cone by the Brownian motion in $\mathbb R^N$ is almost-surely null. Now, the Brownian motion on $\spin^N$ can simply be obtained by projecting a $\mathbb R^N$-Brownian motion on $\spin^N$. Therefore, if started at the extremity of a cone in $\spin^N$, the Brownian motion on $\spin^N$  will also enter the finite cone immediatly. This implies in particular that if $A^c$ satisfies the Poincar\'e cone condition, the boundary of $A^c$ is regular for $\mathring A$ which is the assumption of Proposition \ref{th:asselah}.
\end{proof}

%\section{Alternative proof of Proposition \ref{prop:estimateProbWeqk}}
%Should we add it? -No.

\section*{Acknowledgments}
We thank Ofer Zeitouni for suggesting to consider properties of the winding number that lead to Theorem \ref{thm:cltforw}.

\bibliographystyle{plain}
\bibliography{topological_metastability}

\vspace{20pt}

\end{document}